\documentclass[11pt]{amsart}

\usepackage[a4paper,margin=3cm]{geometry}
\usepackage{amsmath,amssymb,amsthm,mathrsfs}

\usepackage[T1]{fontenc}
\usepackage[utf8]{inputenc}
\usepackage{enumitem}
\usepackage{hyperref}
\usepackage{cleveref}
\usepackage{tikz}
\usetikzlibrary{arrows.meta, positioning}
\usepackage{url}

\newtheorem{theorem}{Theorem}[section]
\newtheorem{lemma}[theorem]{Lemma}
\newtheorem{proposition}[theorem]{Proposition}
\newtheorem{corollary}[theorem]{Corollary}

\theoremstyle{definition}
\newtheorem{definition}[theorem]{Definition}
\newtheorem{remark}[theorem]{Remark}
\newtheorem{example}[theorem]{Example}

\newcommand{\twoheaduparrow}{\mathbin{\uparrow\!\uparrow}}

\numberwithin{equation}{section}

\title{Topological and differentiable aspects of Clifford semigroups}
\author{S. Bonzio, A. Loi and G. Zecchini}
\date{}
\address{Department of Mathematics and Computer Science, University of Cagliari, Italy}
\email{stefano.bonzio@unica.it}


\begin{document}

\maketitle

\begin{abstract}
This paper investigates the interplay between algebraic structure, topology, and differentiability in Clifford semigroups. The study is developed along three main themes. First, in the compact Hausdorff setting, we provide an explicit construction of a compatible metric for the Bowman topology. Second, we address Hilbert-fifth-type questions by establishing criteria under which the maximal subgroups are forced to be Lie groups. Finally, we prove a structural rigidity theorem: $C^1$-regularity at the idempotents implies that the idempotent semilattice is discrete.
\end{abstract}
\section{Introduction}

Clifford semigroups form a natural class of inverse semigroups at the
interface between semigroup theory and topological algebra. Algebraically,
every Clifford semigroup admits a decomposition as a strong semilattice of
groups:
\[
S=\bigsqcup_{e\in E(S)} G_e,
\]
where $E(S)$ is the semilattice of idempotents and $G_e$ is the maximal
subgroup of $S$ containing $e$. This decomposition makes Clifford semigroups a
convenient framework for studying how global topological properties are
governed by the structure of the semilattice of idempotents and by the family of
maximal subgroups.

In the compact Hausdorff setting, several natural topologies on a Clifford
semigroup arise from the semilattice $E(S)$ and the associated functor
$e\mapsto G_e$. In particular, Yeager \cite{Yeager} and Bowman \cite{Bowman} introduced canonical
topologies adapted to this decomposition and proved, under suitable
assumptions, that they recover the original topology of a compact Clifford
semigroup. A natural problem is therefore to determine when these topologies
are metrizable and, more importantly, whether one can construct explicit
compatible metrics.

A second theme of the paper concerns Hilbert-fifth-type questions for
topological Clifford semigroups; the relevant classical background is recalled
in Section~\ref{sec:hilbert}. Since a Clifford semigroup is generally far from
being a manifold, or even a Lie object, globally, the natural question is not
whether $S$ itself is a Lie semigroup, but rather under which topological
assumptions the maximal subgroups $G_e$ must be Lie groups. We shall say that a
topological Clifford semigroup is of Lie type if each of its maximal subgroups
is a Lie group.

A third theme is a rigidity phenomenon at the level of idempotents. We show
that a weak differentiability assumption near idempotents has strong global
consequences for the semigroup structure: namely, $C^1$-regularity at the
idempotents forces the idempotent semilattice to be discrete, and hence forces
the semigroup to split as a strong semilattice of topological groups.

The paper develops these three aspects in a unified framework.

\medskip

Our first main result concerns the construction of an explicit metric compatible with the Bowman
topology.

\begin{theorem}\label{thm:intro-bowman}
Let $S$ be a compact Hausdorff topological Clifford semigroup such that
$E(S)$ is a metrizable perfect semilattice and each maximal subgroup $G_e$ is
compact metrizable. Assume moreover that the associated functor $e\mapsto G_e$
is inverse-limit preserving. Then it is possible to construct an explicit metric on $S$ compatible with the Bowman topology.
\end{theorem}

The metrizability of a topological Clifford semigroup is a problem successfully solved in \cite{MetrizabilityClifford} (see also  \cite{BanakhCompact}), however the solution is typically non-constructive, that is, the existence of a metric is provided but with no explicit form. 

The construction combines a metric on the idempotent semilattice with a
countable family of weighted metrics of the projections of elements into
lower groups $G_b$, where $b$ ranges over a countable domain-theoretic basis of
$E(S)$. The explicit metric is introduced in Section~\ref{sec:metric}; see in
particular \eqref{eq:bowman-metric}, Lemma~\ref{lem:d-metric},
Theorem~\ref{th:metric-equals-bowman}, and
Corollary~\ref{cor:bowman-metrizable-explicit}.

\medskip

Our next main result concerns Hilbert-fifth-type criteria for topological
Clifford semigroups.

\begin{theorem}\label{thm:intro-hilbert}
Let $S$ be a topological Clifford semigroup.
\begin{enumerate}
\item If $S$ is weakly locally Euclidean at the idempotents and is a strong
semilattice of topological groups, then each maximal subgroup $G_e$ is a Lie
group.
\item If $S$ is locally compact and NSS, then each maximal subgroup $G_e$ is a
Lie group.
\item If $S$ is weakly locally Euclidean at the idempotents and each maximal
subgroup $G_e$ is locally arcwise connected, then each maximal subgroup $G_e$
is a Lie group.
\end{enumerate}
In particular, in all cases each maximal subgroup of $S$ is a Lie group; that
is, $S$ is of Lie type.
\end{theorem}

These results are proved in Section~\ref{sec:hilbert}; see
Theorems~\ref{th:hilbert-clifford}, \ref{th:nss-clifford},
and~\ref{th:arc-connected-clifford}.

\medskip

The main rigidity result of the paper is the following.

\begin{theorem}\label{thm:intro-c1}
Let $S$ be a topological Clifford semigroup which is $C^1$ at the idempotents.
Then the semilattice $E(S)$ is discrete. In particular, $S$ is a strong
semilattice of topological groups.
\end{theorem}

This rigidity theorem is established in Section~\ref{sec:c1}; see
Theorem~\ref{th:c1-idempotents-discrete}.

As a consequence of this rigidity theorem and the Lie-type criteria proved in
Section~\ref{sec:hilbert}, we obtain the following corollary.

\begin{corollary}\label{cor:intro-c1-lie}
Let $S$ be a topological Clifford semigroup which is $C^1$ at the idempotents
and weakly locally Euclidean at the idempotents. Then $S$ is of Lie type.
\end{corollary}

This is proved in Corollary~\ref{cor:c1-plus-wle-lie}.

In the finite-dimensional-at-idempotents setting, this rigidity becomes a
characterization.

\begin{theorem}\label{thm:intro-finite}
For a topological Clifford semigroup $S$, the following are equivalent:
\begin{enumerate}
\item $S$ is $C^1$ at the idempotents with finite-dimensional charts;
\item $S$ is a strong semilattice of finite-dimensional Lie groups.
\end{enumerate}
\end{theorem}

The precise characterization is proved in
Theorem~\ref{th:c1-idempotents-characterization}.

\medskip

Taken together, these results show that the idempotent semilattice governs
both the global topology and the local regularity of a topological Clifford
semigroup. In particular, differentiability at idempotents imposes strong
global restrictions on the semigroup structure.

\medskip

The paper is organized as follows. In Section~\ref{sec:preliminaries} we recall
the algebraic and topological background on Clifford semigroups. In
Section~\ref{sec:topological} we discuss and compare the main
topologies arising on Clifford semigroups. In Section~\ref{sec:metric} we study metrizability and construct an
explicit metric for the Bowman topology (Theorem \ref{thm:intro-bowman}). In Section~\ref{sec:hilbert} we prove
the Hilbert-fifth-type results (Theorem \ref{thm:intro-hilbert}). Finally, in Section~\ref{sec:c1} we study
$C^1$-regularity at idempotents and derive the resulting rigidity phenomena (Theorems \ref{thm:intro-c1} and \ref{thm:intro-finite}).

\section{Preliminaries}\label{sec:preliminaries}

\subsection{Inverse and Clifford semigroups}

We briefly recall the basic algebraic notions that will be used throughout the
paper; standard references are \cite{Howie,Grillet}.

In an arbitrary semigroup $S$, we denote by $E(S)$ the set of idempotents of
$S$, that is,
\[
E(S):=\{e\in S : e^2=e\}.
\]

\begin{definition}[Inverse semigroup]
An \emph{inverse semigroup} is a semigroup $S$ such that for every $x\in S$
there exists a unique element $x^{-1}\in S$ satisfying
\[
xx^{-1}x=x,
\qquad
x^{-1}xx^{-1}=x^{-1}.
\]
\end{definition}

In an inverse semigroup the idempotents commute pairwise, hence $E(S)$ is a
semilattice with respect to the multiplication. Moreover, for every $x,y\in S$,
one has
\[
(xy)^{-1}=y^{-1}x^{-1}.
\]

\begin{proposition}\label{prop:clifford-equivalent}
Let $S$ be an inverse semigroup. The following conditions are equivalent:
\begin{enumerate}
\item $xx^{-1}=x^{-1}x$ for every $x\in S$;
\item every idempotent of $S$ is central.
\end{enumerate}
\end{proposition}

\begin{proof}
This is standard (see, for instance, \cite[Chapter~5]{Howie}).
\end{proof}

\begin{definition}[Clifford semigroup]
An inverse semigroup satisfying the equivalent conditions of
Proposition~\ref{prop:clifford-equivalent} is called a
\emph{Clifford semigroup}.
\end{definition}

In a Clifford semigroup it is convenient to write
\[
x^0:=xx^{-1}=x^{-1}x\in E(S).
\]

\begin{lemma}\label{lem:maximal-subgroup}
Let $S$ be a Clifford semigroup and let $e\in E(S)$. Then
\[
G_e:=\{x\in S : x^0=e\}
\]
is a group with identity $e$.
\end{lemma}

\begin{proof}
Let $e\in E(S)\cap G_e$. For every $x\in G_e$ one has
\[
ex=xe=x,
\]
because $e=x^0$ is central, hence $G_{e}$ is a monoid (with identity $e$).
If $x,y\in G_e$, then
\[
(xy)^0=xy(xy)^{-1}=xyy^{-1}x^{-1}=xx^{-1}yy^{-1}=ee=e,
\]
so $xy\in G_e$. If $x\in G_e$, then
\[
(x^{-1})^0=x^{-1}x=e,
\]
hence $x^{-1}\in G_e$. Therefore $G_e$ is a group.
\end{proof}

The structure of a Clifford semigroup is completely determined by the
semilattice $E(S)$, the groups $G_e$, and the bonding homomorphisms
\[
\varphi_{f,e}\colon G_f\to G_e,
\qquad
\varphi_{f,e}(x):=ex=xe,
\qquad (e\leq f),
\]
where the order on $E(S)$ is defined by
\[
e\leq f \quad \Longleftrightarrow \quad ef=e.
\]
Thus the first index of $\varphi_{f,e}$ denotes the domain (group) and the second
index denotes the codomain (group).

\begin{remark}[Clifford semigroups as strong semilattices of groups]
\label{rem:clifford-strong-semilattice}
Let $S$ be a Clifford semigroup. Then:
\begin{enumerate}
\item $E(S)$ is a semilattice;
\item for every $e\in E(S)$, the set $G_e$ is a group;
\item for every $e\leq f$, the map
\[
\varphi_{f,e}\colon G_f\to G_e,\qquad \varphi_{f,e}(x)=ex,
\]
is a group homomorphism;
\item the bonding homomorphisms satisfy
\[
\varphi_{e,e}=\mathrm{id}_{G_e},
\qquad
\varphi_{g,e}=\varphi_{f,e}\circ\varphi_{g,f}
\quad\text{whenever } e\leq f\leq g;
\]
\item one has the disjoint decomposition
\[
S=\bigsqcup_{e\in E(S)} G_e,
\]
and the multiplication in $S$ is given by
\[
st=\varphi_{e,ef}(s)\cdot \varphi_{f,ef}(t),
\qquad s\in G_e,\ t\in G_f,
\]
where the product on the right-hand side is computed in the group $G_{ef}$.
\end{enumerate}
Thus every Clifford semigroup is a strong semilattice of groups.
Conversely, every strong semilattice of groups gives rise to a Clifford
semigroup via the above construction.
\end{remark}

\begin{proposition}\label{prop:group-characterization}
Let $S$ be a Clifford semigroup. The following are equivalent:
\begin{enumerate}
\item $S$ is a group;
\item $S$ has a unique idempotent;
\item for every $x,y\in S$, one has $xyy^{-1}=x$.
\end{enumerate}
\end{proposition}

\begin{proof}
The equivalence between (1) and (2) is immediate.

Assume (2). If $e$ is the unique idempotent of $S$, then $S=G_e$, so for all
$x,y\in S$ we have
\[
xyy^{-1}=xe=x.
\]

Assume now (3). Let $x\in G_e$ and $y\in G_f$. Then
\[
x=xyy^{-1}=xf=\varphi_{e,ef}(x),
\]
and similarly
\[
y=\varphi_{f,ef}(y).
\]
Since the groups $G_e$ are pairwise disjoint, this forces $e=ef=f$, hence
$e=f$. Therefore $E(S)$ is a singleton, so $S$ is a group by (2).
\end{proof}

\begin{definition}[Trivial Clifford semigroup]
\label{def:trivial-clifford}
Let $S$ be a Clifford semigroup, written as
\[
S=\bigsqcup_{e\in E(S)} G_e
\]
with bonding homomorphisms
\[
\varphi_{f,e}\colon G_f\to G_e
\qquad (e\leq f),
\]
as in Remark~\ref{rem:clifford-strong-semilattice}.
We say that $S$ is a \emph{trivial Clifford semigroup} if there exist a group
$G$ and group isomorphisms
\[
\theta_e\colon G_e\to G
\qquad (e\in E(S))
\]
such that for every $e\leq f$ one has
\[
\varphi_{f,e}=\theta_e^{-1}\theta_f.
\]
Equivalently, after identifying each $G_e$ with $G$ via $\theta_e$, all the
bonding homomorphisms become the identity map on $G$.
\end{definition}

\begin{remark}\label{rem:trivial-clifford-product}
If $S$ is a trivial Clifford semigroup, then $S$ is isomorphic, as a Clifford
semigroup, to the direct product $E(S)\times G$, where the multiplication is
given by
\[
(e,g)\cdot(f,h)=(ef,gh),
\qquad
(e,g)^{-1}=(e,g^{-1}).
\]
Conversely, for every semilattice $E$ and every group $G$, the direct product
$E\times G$ with the above operations is a Clifford semigroup.
\end{remark}

Recall that a semigroup $S$ is \emph{left (right) cancellative} if it satisfies the left (right) cancellation laws and \emph{cancellative} if it does satisfy both.

\begin{proposition}\label{prop:cancellative-i-group}
Let $S$ be a left (or right) cancellative inverse semigroup. Then $S$ is a group.
\end{proposition}
\begin{proof}
Let $x,y\in S$. Since $S$ is inverse, one has
\[ x^{-1}y = x^{-1}xx^{-1}y \text{ and } xy = xx^{-1}xy,
\]
hence, by left cancellativity
\begin{equation}\label{eq: cancellazione a sinistra}
  y = xx^{-1}y = x^{-1}xy.
 \end{equation}
Therefore  
\[ xx^{-1}yy^{-1} = x^{-1}xyy^{-1}.\]
By commutativity of the idempotents one gets 
\[ yy^{-1}xx^{-1} = yy^{-1}x^{-1}x,\]
 hence, by left cancellativity (applied twice) 
 \[
 xx^{-1}= x^{-1}x.
 \]
 This shows that $S$ is a Clifford semigroup which, moreover, satisfies $y = xx^{-1}y$ (by equation \eqref{eq: cancellazione a sinistra}) for every $x,y\in S$. Since idempotents are central we get
 \[
 y = y xx^{-1},
 \]
 hence, by Proposition \ref{prop:group-characterization}, $S$ is a group.
%
\end{proof}

\subsection{Topological Clifford semigroups}

\begin{definition}[Topological semigroup]
A \emph{topological semigroup} is a semigroup $(S,\cdot)$ endowed with a
topology such that the multiplication
\[
m\colon S\times S\to S,\qquad m(x,y)=xy,
\]
is continuous.
\end{definition}

\begin{definition}[Topological inverse semigroup]
A \emph{topological inverse semigroup} is an inverse semigroup $S$ endowed with
a topology such that both multiplication and inversion
\[
c\colon S\to S,\qquad c(x)=x^{-1},
\]
are continuous.
\end{definition}

\begin{definition}[Topological Clifford semigroup]
A \emph{topological Clifford semigroup} is a topological inverse semigroup whose
underlying inverse semigroup is a Clifford semigroup.
\end{definition}

If $S$ is a topological Clifford semigroup, then each maximal subgroup $G_e$,
endowed with the subspace topology, is a topological group, and each bonding
homomorphism $\varphi_{f,e}$ is continuous.

\begin{definition}[Trivial topological Clifford semigroup]
\label{def:trivial-topological-clifford}
A \emph{trivial topological Clifford semigroup} is a topological Clifford
semigroup $(S,\tau)$ such that:
\begin{enumerate}
\item the underlying Clifford semigroup is trivial in the sense of
Definition~\ref{def:trivial-clifford};
\item under some identification $S\cong E(S)\times G$, the topology $\tau$
coincides with the product topology.
\end{enumerate}
Equivalently, $(S,\tau)$ is topologically and algebraically isomorphic to a
direct product $E\times G$ of a topological semilattice $E$ and a topological
group $G$.
\end{definition}

\begin{example}\label{ex:algebraically-trivial-not-topologically}
A Clifford semigroup may be trivial algebraically without being trivial
topologically. Let
\[
G=\mathbb Z_2=\{\overline 0,\overline 1\}
\]
and
\[
E=\{0\}\cup\{1/n:n\in\mathbb N\}\subseteq \mathbb R,
\]
endowed with the semilattice operation
\[
xy=
\begin{cases}
x,& x=y,\\
0,& x\neq y.
\end{cases}
\]
Consider the trivial Clifford semigroup $S=E\times G$. Give
$E\times\{\overline 0\}$ the usual convergent-sequence topology and
$E\times\{\overline 1\}$ the discrete topology. Then $S$ is a metrizable
topological Clifford semigroup whose underlying Clifford semigroup is
algebraically trivial, but it is not homeomorphic to $E\times G$ with the
product topology.
\end{example}

\subsection{The idempotent map}

Let $S$ be a topological inverse semigroup. Define
\[
\pi\colon S\to E(S),
\qquad
\pi(x):=xx^{-1}.
\]

\begin{lemma}\label{lem:pi-continuous}
The map $\pi$ is continuous. Moreover, if $S$ is a topological Clifford
semigroup, then $\pi$ is a semigroup homomorphism and
\[
\pi^{-1}(\{e\})=G_e
\qquad\text{for every } e\in E(S).
\]
\end{lemma}

\begin{proof}
The map
\[
x\longmapsto (x,x^{-1})
\]
from $S$ to $S\times S$ is continuous, and multiplication is continuous, so
\[
\pi(x)=xx^{-1}
\]
is continuous.

Assume now that $S$ is Clifford. Then for every $x,y\in S$,
\[
\pi(xy)=(xy)(xy)^{-1}=xyy^{-1}x^{-1}=xx^{-1}yy^{-1}=\pi(x)\pi(y),
\]
where we used the fact that idempotents are central. Thus $\pi$ is a semigroup
homomorphism.

Finally, by definition,
\[
G_e=\{x\in S:x^0=e\}=\{x\in S:xx^{-1}=e\},
\]
hence $\pi^{-1}(\{e\})=G_e$.
\end{proof}

\begin{proposition}\label{prop:pi-characterizes-clifford}
Let $S$ be an inverse semigroup. Then $S$ is a Clifford semigroup if and only if
\[
\pi(x^{-1})=(\pi(x))^{-1}
\qquad\text{for every } x\in S.
\]
\end{proposition}

\begin{proof}
If $S$ is Clifford, then
\[
\pi(x^{-1})=x^{-1}x=xx^{-1}=(xx^{-1})^{-1}=(\pi(x))^{-1}.
\]

Conversely, if
\[
\pi(x^{-1})=(\pi(x))^{-1},
\]
then
\[
x^{-1}x=(xx^{-1})^{-1}=xx^{-1},
\]
so $S$ satisfies the Clifford identity and is therefore a Clifford semigroup.
\end{proof}

\subsection{Topological semilattices}

We conclude the section with a few facts on topological semilattices that will
be used later.

Throughout the paper, a semilattice is understood as a meet-semilattice. Recall that when $(S,\cdot)$ the following
\[
x\leq y \quad \Longleftrightarrow \quad xy=x,
\]
defines a partial order on $S$ such that $xy$ is the inf of $\{x,y\}$. 

\begin{lemma}\label{lem:hausdorff-semilattice}
Let $S$ be a topological semilattice. Then $S$ is Hausdorff if and only if
\[
\Delta_{\leq}:=\{(x,y)\in S\times S : x\leq y\}
\]
is closed in $S\times S$.
\end{lemma}

\begin{proof}
Assume first that $S$ is Hausdorff. Consider the continuous map
\[
f\colon S\times S\to S\times S,
\qquad
f(x,y)=(xy,x).
\]
Then
\[
\Delta_{\leq}=f^{-1}(\Delta_S),
\]
where $\Delta_S$ is the diagonal of $S\times S$. Since $\Delta_S$ is closed in
a Hausdorff space, $\Delta_{\leq}$ is closed.

Conversely, assume that $\Delta_{\leq}$ is closed. The map
\[
g\colon S\times S\to S\times S,
\qquad
g(x,y)=(y,x),
\]
is a homeomorphism, so
\[
\Delta_{\geq}:=g(\Delta_{\leq})
\]
is also closed. Therefore
\[
\Delta_S=\Delta_{\leq}\cap \Delta_{\geq}
\]
is closed in $S\times S$, and hence $S$ is Hausdorff.
\end{proof}

\begin{lemma}\label{lem:compact-semilattice-minimum}
Let $S$ be a compact Hausdorff semilattice. Then $S$ has a minimum element.
\end{lemma}

\begin{proof}
For each $x\in S$, the down set
\[
\downarrow x:=\{y\in S:y\leq x\}
\]
is closed by Lemma~\ref{lem:hausdorff-semilattice}. Moreover, for every finite
family $x_1,\dots,x_n\in S$, one has
\[
x_1\cdots x_n\in \downarrow x_1\cap \cdots \cap \downarrow x_n,
\]
so the family $\{\downarrow x:x\in S\}$ has the finite intersection property.
By compactness,
\[
\{z\} = \bigcap_{x\in S}\downarrow x\neq\varnothing,
\]
and $z$ is the minimum element in $S$.
\end{proof}

\section{Topologies on Clifford semigroups}\label{sec:topological}

A Clifford semigroup may admit several different compatible topologies making it
into a topological Clifford semigroup. In this section we discuss three natural
topological frameworks that arise in the literature: the topology of strong
semilattices of topological groups, the Yeager topology, and the Bowman
topology.

\subsection{Strong semilattices of topological groups}

We begin with the topology considered in \cite{Abd-Allah,MaityPaul-Semilattice},
which turns a Clifford semigroup into a strong semilattice of topological
groups.

\begin{definition}[Strong semilattice of topological groups]
\label{def:strong-semilattice-top-groups}
Let
\[
S=\bigsqcup_{e\in E(S)} G_e
\]
be a Clifford semigroup, written as a strong semilattice of groups as in
Remark~\ref{rem:clifford-strong-semilattice}. We say that $S$ is a
\emph{strong semilattice of topological groups} if each $G_e$ is endowed with a
topology making it a topological group, each bonding homomorphism
\[
\varphi_{f,e}\colon G_f\to G_e
\qquad (e\leq f)
\]
is continuous, and $S$ is endowed with the disjoint union topology generated by
the family of topologies of the groups $G_e$.
\end{definition}

In this situation the maximal subgroups are open by construction, and the
topology on $S$ is exactly the topological sum topology of the family
$\{G_e\}_{e\in E(S)}$.

\begin{definition}[MP property]
\label{def:mp}
Let $(S,\tau)$ be a topological semigroup. We say that $S$ has the
\emph{MP property}\footnote{We choose this nomenclature from and Maity and Paul \cite{MaityPaul-Semilattice}.} if for every open set $O\in\tau$ and every $x\in O$ there
exists an open neighbourhood $U$ of $x$ such that
\[
U\subseteq O\cap J_x,
\]
where $J_x$ denotes the $\mathcal J$-class of $x$.
\end{definition}

\begin{lemma}\label{lem:mp-j-open}
Let $S$ be a topological semigroup. Then $S$ has the MP property if and only if
every $\mathcal J$-class of $S$ is open.
\end{lemma}

\begin{proof}
Assume first that $S$ has the MP property. Fix $x\in S$ and let $y\in J_x$.
Applying the MP property to the open set $S$ and the point $y$, we find an open
set $U$ such that
\[
y\in U\subseteq J_y=J_x.
\]
Thus $J_x$ is open.

Conversely, if every $\mathcal J$-class is open and $O$ is open with $x\in O$,
then
\[
U:=O\cap J_x
\]
is an open neighbourhood of $x$ contained in $O\cap J_x$. Hence $S$ has the MP
property.
\end{proof}

Recall that in a Clifford semigroup, the $\mathcal J$-classes coincide with the maximal
subgroups $G_e$.

\begin{remark}\label{rem:mp-clifford}
Let $S$ be a topological Clifford semigroup. Since in a Clifford semigroup the
$\mathcal J$-classes are precisely the maximal subgroups, one has
\[
S \text{ has the MP property}
\quad\Longleftrightarrow\quad
G_e \text{ is open in } S \text{ for every } e\in E(S).
\]
\end{remark}

The next proposition gathers several equivalent formulations that will be used
repeatedly later.

\begin{proposition}\label{prop:mp-equivalences}
Let $S$ be a topological Clifford semigroup. The following are equivalent:
\begin{enumerate}
\item $S$ is a strong semilattice of topological groups;
\item $S$ has the MP property;
\item every $\mathcal J$-class of $S$ is open;
\item every maximal subgroup $G_e$ is open in $S$;
\item $E(S)$ is discrete in $S$;
\item the canonical bijection
\[
\bigsqcup_{e\in E(S)} G_e \longrightarrow S
\]
is a homeomorphism.
\end{enumerate}
\end{proposition}

\begin{proof}
The equivalence between (1) and (2) is proved in
\cite[Theorem~2.19]{MaityPaul-Semilattice}.
The equivalence (2)$\Leftrightarrow$(3) is Lemma~\ref{lem:mp-j-open}.
The equivalence (3)$\Leftrightarrow$(4) follows from the fact that in a Clifford
semigroup the $\mathcal J$-classes coincide with the maximal subgroups.
Assume (4). Fix $e\in E(S)$. Since the only idempotent in the group $G_e$ is
its identity $e$, we have
\[
E(S)\cap G_e=\{e\}.
\]
Since $G_e$ is open in $S$, the singleton $\{e\}$ is open in the subspace
topology of $E(S)$. Hence $E(S)$ is discrete. Thus (4)$\Rightarrow$(5).
Assume (5). By Lemma~\ref{lem:pi-continuous}, the map
\[
\pi\colon S\to E(S),\qquad \pi(x)=xx^{-1},
\]
is continuous, and
\[
\pi^{-1}(\{e\})=G_e
\qquad\text{for every } e\in E(S).
\]
Since $E(S)$ is discrete, each singleton $\{e\}$ is open in $E(S)$, hence
$G_e$ is open in $S$. Thus (5)$\Rightarrow$(4).
Assume (4). Then the family $\{G_e\}_{e\in E(S)}$ is a pairwise disjoint open
cover of $S$, and the topology on $S$ is exactly the disjoint union topology.
Thus the canonical bijection
\[
\bigsqcup_{e\in E(S)} G_e \longrightarrow S
\]
is a homeomorphism. Hence (4)$\Rightarrow$(6).
Finally, in the disjoint union topology each component $G_e$ is open, so
(6)$\Rightarrow$(4).
\end{proof}

\begin{remark}\label{rem:different-topologies}
Let $S$ be a Clifford semigroup. Even if each maximal subgroup $G_e$ is endowed
with a topology making it a topological group and all bonding homomorphisms are
continuous, the resulting topology on $S$ need not be unique. In particular, a
Clifford semigroup that is algebraically a direct product $E(S)\times G$ need
not be a trivial topological Clifford semigroup in the sense of
Definition~\ref{def:trivial-topological-clifford} (see
Example~\ref{ex:algebraically-trivial-not-topologically}). This is one of the
reasons why the Yeager and Bowman constructions are relevant: they provide
canonical topologies associated with the semilattice and the family of maximal
subgroups.
\end{remark}

\subsection{The Yeager topology}\label{subsec: Yeager}

The decomposition of Clifford semigroups as strong semilattice of groups yields a functor
\[
\mathscr G\colon E(S)\to\mathcal G,
\]
where $E(S)$ is viewed as a small category and $\mathcal G$ denotes the
category of groups.

Throughout this subsection we assume that $E(S)$ is endowed with a compact
Hausdorff topology $\mathcal T_{E(S)}$, each $G_e$ is endowed with a compact
Hausdorff group topology $\mathcal T_e$, and the functor $\mathscr G$ is
inverse-limit preserving in the sense of \cite{Yeager}.

For $U\in\mathcal T_{E(S)}$, $e\in U$, and $V\in\mathcal T_e$, define
\begin{equation}\label{eq:W-yeager}
W(U,(e,V))
:=
\bigcup_{f\in U\cap\uparrow e}\varphi_{f,e}^{-1}(V)
\subseteq S.
\end{equation}

\begin{definition}[Yeager topology \cite{Yeager}]
\label{def:yeager-topology}
A subset $O\subseteq S$ is \emph{open in the Yeager topology} $\mathcal T_Y$ if
for every $(e,s)\in O$ there exists an open neighbourhood $U\in\mathcal
T_{E(S)}$ of $e$ such that for every $f\in U\cap\downarrow e$ there exists
$V_f\in\mathcal T_f$ satisfying:
\begin{enumerate}
\item
$(e,s)\in W(U,(f,V_f))\subseteq O$;
\item whenever $(e',s')\in W(U,(f,V_f))$ and $g\in U\cap\downarrow f$, there
exists $V_g\in\mathcal T_g$ such that
\[
(e',s')\in W(U,(g,V_g))\subseteq O.
\]
\end{enumerate}
\end{definition}

When the semilattice is first countable, Yeager showed that this description
simplifies (see Remark \ref{rem:yeager-inclusion}).

\begin{definition}[First-countable Yeager topology]
\label{def:yeager-first-countable}
Assume that $(E(S),\mathcal T_{E(S)})$ is first countable. The
\emph{first-countable Yeager topology} $\mathcal T_{Y_1}$ is obtained from
Definition~\ref{def:yeager-topology} by requiring only condition~(1).
\end{definition}

\begin{remark}\label{rem:yeager-inclusion}
One always has
$\mathcal T_Y\subseteq \mathcal T_{Y_1}$.
Moreover, if $\mathcal T_Y$ is Hausdorff, then
$\mathcal T_Y=\mathcal T_{Y_1}$
by \cite{Yeager}.
\end{remark}

The relevance of this topology comes from Yeager's reconstruction theorem.

\begin{theorem}[Yeager \cite{Yeager}]\label{th:yeager}
Let $S$ be a compact Hausdorff topological Clifford semigroup. Endow $E(S)$ and
each maximal subgroup $G_e$ with the subspace topologies induced by $S$. Then:
\begin{enumerate}
\item $E(S)$ is a compact Hausdorff topological semilattice;
\item each $G_e$ is a compact Hausdorff topological group;
\item the functor $\mathscr G\colon E(S)\to\mathcal G$ is inverse-limit
preserving;
\item the topology of $S$ coincides with $\mathcal T_Y$;
\item if $E(S)$ is first countable, then the topology of $S$ also coincides
with $\mathcal T_{Y_1}$.
\end{enumerate}
\end{theorem}

\subsection{The Bowman topology}

We now turn our attention to the Bowman topology, which is defined under stronger assumptions
on the semilattice of idempotents and is particularly well suited for
metrizability questions.

We first recall the basic domain-theoretic notions that will be used.

\begin{definition}\cite[Definition~I-1.1]{Gierzetal.}\label{def:way-below}
Let $(P,\leq)$ be a poset. For $x,y\in P$, one says that $x$ is
\emph{way below} $y$, and writes $x\ll y$, if for every nonempty up-directed
subset $D\subseteq P$ such that $\sup D$ exists and $y\leq \sup D$, there
exists $d\in D$ with $x\leq d$.
\end{definition}

\begin{lemma}\cite[Proposition~I-1.2]{Gierzetal.}\label{lem:way-below-basic}
Let $(P,\leq)$ be a poset. Then:
\begin{enumerate}
\item if $x\ll y$, then $x\leq y$;
\item if $P$ has a minimum element $0$, then $0\ll x$ for every $x\in P$.
\end{enumerate}
\end{lemma}

\begin{definition}\label{def:domain}
A poset $(P,\leq)$ is a \emph{dcpo} if every nonempty up-directed subset
$D\subseteq P$ has a supremum. A dcpo is a \emph{domain} if it is continuous,
that is, for every $x\in P$ one has
\[
x=\sup\{y\in P:y\ll x\}.
\]
\end{definition}

\begin{definition}\cite[Chapter~III]{Gierzetal.}\label{def:lawson-topology}
Let $P$ be a domain. The \emph{Lawson topology} on $P$ is the topology
generated by the sets
\[
(\twoheaduparrow x)\setminus(\uparrow F),
\qquad x\in P,\quad F\in\mathcal P_{\mathrm{fin}}(P),
\]
where
\[
\twoheaduparrow x:=\{y\in P:x\ll y\},
\qquad
\uparrow F:=\bigcup_{f\in F}\uparrow f.
\]
\end{definition}

\begin{definition}\cite[Section~2]{Lawson}\label{def:lawson-semilattice}
A Hausdorff topological semilattice $S$ is called a \emph{Lawson semilattice} if
every point of $S$ has a neighbourhood basis consisting of open
subsemilattices.
\end{definition}

\begin{definition}[\cite{Bowman}]\label{def:perfect-semilattice}
A \emph{perfect semilattice} is a compact Hausdorff Lawson semilattice.
\end{definition}

The following characterization connects perfect semilattices with the Lawson
topology.

\begin{theorem}\cite[Theorem~VI-3.4]{Gierzetal.}\label{th:perfect-semilattice-characterization}
For a topological semilattice $S$, the following are equivalent:
\begin{enumerate}
\item $S$ is a perfect semilattice;
\item $S$ is a complete continuous semilattice, and its topology is the Lawson
topology.
\end{enumerate}
\end{theorem}

\begin{corollary}\label{cor:inf-open}
Let $S$ be a perfect semilattice. Then for every nonempty open subset
$U\subseteq S$, the infimum $\inf U$ exists.
\end{corollary}

\begin{proof}
By Theorem~\ref{th:perfect-semilattice-characterization}, $S$ is complete, so
every nonempty subset of $S$ admits an infimum.
\end{proof}

We now define Bowman's basic sets. Let $S$ be as in subsection \ref{subsec: Yeager}, and
assume moreover that $(E(S),\mathcal T_{E(S)})$ is a perfect semilattice.

\begin{definition}[Bowman basic sets]
\label{def:bowman-basic}
For every nonempty open subset $U\subseteq E(S)$ and every open subset
$V\subseteq G_{\inf U}$, define
\[
W_B(U,V):=W\bigl(U,(\inf U,V)\bigr),
\]
where $W(U,(e,V))$ is as in equality \eqref{eq:W-yeager}.
\end{definition}

\begin{theorem}[Bowman \cite{Bowman}]\label{th:bowman}
Assume that:
\begin{enumerate}
\item $(E(S),\mathcal T_{E(S)})$ is a perfect semilattice;
\item each maximal subgroup $G_e$ is a compact Hausdorff topological group;
\item the functor $\mathscr G\colon E(S)\to\mathcal G$ is inverse-limit
preserving.
\end{enumerate}
Then the family
\[
\mathcal B_B
:=
\{W_B(U,V):U\subseteq E(S)\text{ open nonempty},\ V\subseteq G_{\inf U}\text{
open}\}
\]
is a basis for a compact Hausdorff topology $\mathcal T_B$ on $S$.
Moreover, with respect to $\mathcal T_B$, $S$ becomes a topological Clifford semigroup.
\end{theorem}

The Bowman topology agrees with the Yeager topology in the compact Hausdorff
setting whenever the semilattice of idempotents is Lawson.

\begin{corollary}\label{cor:yeager-bowman}
Let $S$ be a compact Hausdorff topological Clifford semigroup. Endow $E(S)$ and
each maximal subgroup $G_e$ with the subspace topologies induced by $S$. If
$E(S)$ is a Lawson semilattice, then
\[
\mathcal T_S=\mathcal T_Y=\mathcal T_B.
\]
Moreover, if $E(S)$ is first countable, then also
\[
\mathcal T_S=\mathcal T_Y=\mathcal T_{Y_1}=\mathcal T_B.
\]
\end{corollary}

\begin{proof}
By Theorem~\ref{th:yeager}, the topology of $S$ coincides with the Yeager
topology, so
\[
\mathcal T_S=\mathcal T_Y.
\]
Assume now that $E(S)$ is a Lawson semilattice. Since $E(S)$ is compact
Hausdorff, it is a perfect semilattice by
Definition~\ref{def:perfect-semilattice}. Hence the hypotheses of
Theorem~\ref{th:bowman} are satisfied, and Bowman’s construction yields a
compact Hausdorff Clifford semigroup topology $\mathcal T_B$ on the same
underlying Clifford semigroup and hence, by Theorem \ref{th:yeager} $\mathcal T_B = \mathcal T_S$.
%
\end{proof}

\section{Metrizability}\label{sec:metric}

The metrizability of topological Clifford semigroups has been studied in
\cite{MetrizabilityClifford,BanakhCompact}. In this section we first discuss
the case of strong semilattices of topological groups, and then turn to the
compact Hausdorff setting, where we construct an explicit compatible metric for
the Bowman topology.

\subsection{Strong semilattices of topological groups}

Throughout this subsection, $S$ denotes a topological Clifford semigroup which
is a strong semilattice of topological groups. The following is the natural analogue of the Birkhoff--Kakutani theorem for
topological groups.

\begin{proposition}\label{prop:bk-semilattice}
Let $S$ be a strong semilattice of topological groups. Then $S$ is metrizable
if and only if it is Hausdorff and first countable.
\end{proposition}

\begin{proof}
The forward implication is clear.
Conversely, if $S$ is Hausdorff and first countable, then each maximal subgroup
$G_e$ is Hausdorff and first countable in the subspace topology. Since $G_e$ is
a topological group, it is metrizable by the Birkhoff--Kakutani theorem.
Because $S$ is the topological sum
\[
S=\bigsqcup_{e\in E(S)} G_e,
\]
it follows that $S$ is metrizable.
\end{proof}

A concrete compatible metric can be described as follows.

\begin{remark}\label{rem:metric-disjoint-sum}
Assume that $S=\bigsqcup_{e\in E(S)} G_e$ is a strong semilattice of
topological groups, and for each $e\in E(S)$ let $d_e$ be a compatible metric
on $G_e$ such that $d_e\leq 1$. Since $E(S)$ is discrete in $S$ by
Proposition~\ref{prop:mp-equivalences}, define
\[
d\bigl((e,s),(f,t)\bigr):=
\begin{cases}
d_e(s,t), & \text{if } e=f,\\[4pt]
1, & \text{if } e\neq f.
\end{cases}
\]
Then $d$ is a compatible metric on $S$.
\end{remark}

We next record a general lemma on topological sums. Recall that a subset $A\subseteq X$ of a topological space is a $G_{\delta}$ set if $A = \displaystyle\bigcap_{n \ge 1} U_{n}$ with $U_n$ open in $X$ and the intersection is countable.

\begin{lemma}\label{lem:gdelta-sum}
Let $\{X_i:i\in I\}$ be a family of topological spaces, and let
\[
X=\bigsqcup_{i\in I} X_i
\]
be their topological sum. For each $i\in I$, choose a point $x_i\in X_i$, and
set
\[
E:=\{x_i:i\in I\}\subseteq X.
\]
Then $E$ is a $G_\delta$-subset of $X$ if and only if, for every $i\in I$, the
singleton $\{x_i\}$ is a $G_\delta$-subset of $X_i$.
\end{lemma}

\begin{proof}
Assume first that
\[
E=\bigcap_{n\geq 1} U_n
\]
for some open subsets $U_n\subseteq X$. Fix $i\in I$. Since $X_i$ is open in
the disjoint union topology,
\[
\{x_i\}=E\cap X_i=\bigcap_{n\geq 1}(U_n\cap X_i),
\]
and each $U_n\cap X_i$ is open in $X_i$. Thus $\{x_i\}$ is a $G_\delta$-subset
of $X_i$.

Conversely, suppose that for each $i\in I$ one has
\[
\{x_i\}=\bigcap_{n\geq 1} V_n^{(i)},
\]
where each $V_n^{(i)}$ is open in $X_i$. Define
\[
U_n:=\bigcup_{i\in I} V_n^{(i)}.
\]
Then each $U_n$ is open in $X$, and for every $i\in I$,
\[
\left(\bigcap_{n\geq 1}U_n\right)\cap X_i
=
\bigcap_{n\geq 1}(U_n\cap X_i)
=
\bigcap_{n\geq 1}V_n^{(i)}
=
\{x_i\}.
\]
Since the $X_i$ are pairwise disjoint and cover $X$, it follows that
\[
\bigcap_{n\geq 1}U_n=E.
\]
Hence $E$ is a $G_\delta$-subset of $X$.
\end{proof}

Applied to Clifford semigroups, this yields:

\begin{lemma}\label{lem:e-gdelta-points}
Let $S$ be a topological Clifford semigroup. Assume that every maximal subgroup
$G_e$ is open in $S$. Then $E(S)$ is a $G_\delta$-subset of $S$ if and only if,
for every $e\in E(S)$, the singleton $\{e\}$ is a $G_\delta$-subset of $G_e$.
\end{lemma}

\begin{proof}
By Proposition~\ref{prop:mp-equivalences}, $S$ is the topological sum
\[
S=\bigsqcup_{e\in E(S)} G_e.
\]
Moreover, $E(S)$ meets each $G_e$ exactly in the singleton $\{e\}$. The claim
therefore follows immediately from Lemma~\ref{lem:gdelta-sum}.
\end{proof}

The following is a convenient reformulation of
\cite[Theorem~3.4]{MetrizabilityClifford} in the present setting.

\begin{proposition}\label{prop:countably-compact-metrizable}
Let $S$ be a countably compact $T_2$ strong semilattice of topological groups.
Then the following are equivalent:
\begin{enumerate}
\item $S$ is metrizable;
\item $E(S)$ is a $G_\delta$-subset of $S$;
\item for every $e\in E(S)$, the singleton $\{e\}$ is a $G_\delta$-subset of
$G_e$.
\end{enumerate}
\end{proposition}

\begin{proof}
The equivalence between (2) and (3) is Lemma~\ref{lem:e-gdelta-points}.
The implication (1)$\Rightarrow$(2) is immediate, since in a metric space every
closed subset is a $G_\delta$-subset.
Assume (2). Since $S$ is a strong semilattice of topological groups, the
subspace $E(S)$ is discrete by Proposition~\ref{prop:mp-equivalences}, hence
metrizable. Therefore \cite[Theorem~3.4]{MetrizabilityClifford} applies and
shows that $S$ is metrizable.
\end{proof}

The next example shows that discreteness of $E(S)$ does not imply that $E(S)$
is a $G_\delta$-subset.

\begin{example}\label{ex:discrete-not-gdelta}
There exists a topological Clifford semigroup $S$ such that $E(S)$ is discrete
in $S$ but is not a $G_\delta$-subset of $S$.
\end{example}

\begin{proof}
Let $E$ be an infinite discrete semilattice, for example
\[
E=(\mathbb N,\min),
\]
endowed with the discrete topology. Let $I$ be an uncountable index set, and
let
\[
T:=\{z\in\mathbb C:|z|=1\}
\]
be the circle group. Consider the compact Hausdorff topological group
\[
H:=T^I
\]
with the product topology. Since $I$ is uncountable, $H$ is not first countable
at any point, and in particular the singleton $\{1_H\}$ is not a
$G_\delta$-subset of $H$.
Now consider the trivial topological Clifford semigroup
\[
S=E\times H.
\]
Its idempotents are precisely the pairs
\[
E(S)=E\times\{1_H\}.
\]
For each $e\in E$, the maximal subgroup at $(e,1_H)$ is
$G_e=\{e\}\times H$.
Since $E$ is discrete, each $G_e$ is open in $S$, and $E(S)$ is discrete in
the subspace topology. However, for each $e\in E$, the singleton
$\{(e,1_H)\}$
is not a $G_\delta$-subset of $G_e\cong H$. Hence, by
Lemma~\ref{lem:e-gdelta-points}, the set $E(S)$ is not a $G_\delta$-subset of
$S$.
\end{proof}

\subsection{An explicit metric for the Bowman topology}

We now focus on the compact Hausdorff setting. Let $S$ be a compact Hausdorff
topological Clifford semigroup, let $E(S)$ be its semilattice of idempotents,
and let
\[
S=\bigsqcup_{e\in E(S)} G_e
\]
be the decomposition into maximal subgroups. Assume:
\begin{enumerate}
\item[\textbf{(B1)}] $E(S)$ is a metrizable perfect semilattice;
\item[\textbf{(B2)}] each maximal subgroup $G_e$ is compact metrizable;
\item[\textbf{(B3)}] the functor
\[
\mathscr G\colon E(S)\to\mathcal G,\qquad e\mapsto G_e,
\]
is inverse-limit preserving.
\end{enumerate}

Under these assumptions, the topology of $S$ coincides with the Bowman topology
by Corollary~\ref{cor:yeager-bowman}. The goal of this subsection is to
construct an explicit compatible metric on $S$.

We begin with a standard domain-theoretic notion.

\begin{definition}\label{def:domain-basis}
Let $(P,\leq)$ be a dcpo. A subset $\mathcal B\subseteq P$ is called a
\emph{basis} if for every $x\in P$ the set
\[
\mathcal B_x:=\{b\in\mathcal B:b\ll x\}
\]
is up-directed and satisfies
\[
x=\sup \mathcal B_x.
\]
\end{definition}

\begin{theorem}\cite[Theorem~III-4.5]{Gierzetal.}\label{th:countable-basis-lawson}
Let $P$ be a domain. Then $P$ admits a countable basis if and only if it is
second countable with respect to the Lawson topology.
\end{theorem}

Fix a compatible metric $\rho$ on $E(S)$, bounded by $1$:
\[
0\leq \rho\leq 1.
\]
Since $E(S)$ is compact metrizable, it is second countable. By
Theorem~\ref{th:perfect-semilattice-characterization} and
Theorem~\ref{th:countable-basis-lawson}, the Lawson topology on $E(S)$ admits a
countable domain basis. Fix such a basis and enumerate it as
\[
\mathcal B=\{b_j\}_{j\geq 1}.
\]

For each $b\in\mathcal B$, choose:
\begin{enumerate}
\item a compatible metric $d_b$ on $G_b$ with $0\leq d_b\leq 1$;
\item a base point $c_b\in G_b$;
\item a countable dense subset
\[
D_b=\{t_{b,k}\}_{k\geq 1}\subseteq G_b.
\]
\end{enumerate}

The next lemma is the key separation fact.

\begin{lemma}\label{lem:separate-via-basis}
Assume that $\mathscr G$ is inverse-limit preserving. Let $e\in E(S)$ and let
$g,h\in G_e$ with $g\neq h$. Then there exists $b\in\mathcal B$ such that
\[
b\ll e
\qquad\text{and}\qquad
\varphi_{e,b}(g)\neq \varphi_{e,b}(h).
\]
\end{lemma}

\begin{proof}
Let
\[
D_e:=\{b\in\mathcal B:b\ll e\}.
\]
By Definition~\ref{def:domain-basis}, $D_e$ is up-directed and
\[
e=\sup D_e.
\]
Since $\mathscr G$ is inverse-limit preserving, the canonical map
\[
\eta_e\colon G_e\to \varprojlim_{b\in D_e} G_b,
\qquad
\eta_e(g)=\bigl(\varphi_{e,b}(g)\bigr)_{b\in D_e},
\]
is injective by inverse-limit preservation; see \cite[Lemma~1]{Yeager}. 

If
\[
\varphi_{e,b}(g)=\varphi_{e,b}(h)
\qquad\text{for all } b\in D_e,
\]
then $\eta_e(g)=\eta_e(h)$, contradicting $g\neq h$. Hence there exists
$b\in D_e$ such that
\[
\varphi_{e,b}(g)\neq \varphi_{e,b}(h).
\]
\end{proof}

The idea of the construction is to combine the metric structure of the
semilattice $E(S)$ with a family of metrics in the groups $G_b$,
where $b$ ranges over a countable domain basis. The metric measures both
the distance between idempotents and the discrepancy between the
projections of elements along the inverse system defined by the bonding maps.

For each $b\in E(S)$, define
\[
\twoheaduparrow b:=\{e\in E(S): b\ll e\}.
\]
Since $E(S)$ carries the Lawson topology, the set $\twoheaduparrow b$ is open.
For $b\in\mathcal B\setminus\{0\}$, define
\[
a_b(e):=\rho\bigl(e,E(S)\setminus\twoheaduparrow b\bigr),
\qquad e\in E(S),
\]
and set $a_0\equiv 1$.

\begin{lemma}\label{lem:ab}
For each $b\in\mathcal B$, the function
\[
a_b\colon E(S)\to[0,1]
\]
is $1$-Lipschitz. Moreover, for $b\neq 0$,
\[
a_b(e)>0
\quad\Longleftrightarrow\quad
b\ll e.
\]
In particular, if $b\not\leq e$, then $b\not\ll e$ and hence $a_b(e)=0$.
\end{lemma}

\begin{proof}
Since $a_b$ is the distance to a closed subset of a metric space, it is
$1$-Lipschitz. The second claim follows from the fact that
$\twoheaduparrow b$ is open. Finally, if $b\not\leq e$, then
$b\not\ll e$ by Lemma~\ref{lem:way-below-basic}, so $a_b(e)=0$.
\end{proof}

For $e\in E(S)$ and $b\in\mathcal B$, define the extended bonding map
\[
\widehat\varphi_{e,b}\colon G_e\to G_b
\]
by
\[
\widehat\varphi_{e,b}=
\begin{cases}
\varphi_{e,b}, & \text{if } b\leq e,\\[4pt]
\text{the constant map } c_b, & \text{if } b\not\leq e.
\end{cases}
\]

For $b\in\mathcal B$ and $(e,g),(f,h)\in S$, define
\begin{align}
P_b\bigl((e,g),(f,h)\bigr)
&:=
|a_b(e)-a_b(f)|
\notag\\
&\quad
+\sum_{k\geq 1}2^{-k}
\left|
a_b(e)\,d_b\bigl(\widehat\varphi_{e,b}(g),t_{b,k}\bigr)
-
a_b(f)\,d_b\bigl(\widehat\varphi_{f,b}(h),t_{b,k}\bigr)
\right|.
\label{eq:Pb}
\end{align}

Finally, define
\begin{equation}\label{eq:bowman-metric}
d\bigl((e,g),(f,h)\bigr)
:=
\rho(e,f)+\sum_{j\geq 1}2^{-j}P_{b_j}\bigl((e,g),(f,h)\bigr).
\end{equation}

\begin{lemma}\label{lem:d-metric}
The function $d$ defined in \eqref{eq:bowman-metric} is a metric on $S$.
Moreover,
$0\leq d\leq 3$.
\end{lemma}

\begin{proof}
For each $b\in\mathcal B$, since $0\leq a_b\leq 1$ and $0\leq d_b\leq 1$, the
series defining $P_b$ converges and satisfies
$0\leq P_b\leq 2$.
Hence the series in \eqref{eq:bowman-metric} converges absolutely and uniformly,
so $d$ is well defined and bounded by $3$.
Symmetry and non-negativity are immediate. 
For each fixed $b\in\mathcal B$, the map
\[
(e,g)\longmapsto
\left(
a_b(e),
\bigl(a_b(e)\,d_b(\widehat\varphi_{e,b}(g),t_{b,k})\bigr)_{k\ge1}
\right)
\]
takes values in $\mathbb R\times \ell^1(\mathbb N)$, and $P_b$ is exactly the $\ell^1$-distance between the corresponding images. Hence $P_b$ satisfies the triangle inequality.
It remains to prove that
\[
d\bigl((e,g),(f,h)\bigr)=0
\quad\Longrightarrow\quad
(e,g)=(f,h).
\]
If $d((e,g),(f,h))=0$, then $\rho(e,f)=0$, hence $e=f$.
Assume now that $g\neq h$. By Lemma~\ref{lem:separate-via-basis}, there exists
$b\in\mathcal B$ such that
\[
b\ll e
\qquad\text{and}\qquad
\varphi_{e,b}(g)\neq \varphi_{e,b}(h).
\]
Since $b\ll e$, Lemma~\ref{lem:ab} gives $a_b(e)>0$. Therefore
\[
P_b\bigl((e,g),(e,h)\bigr)=0
\]
implies
\[
d_b\bigl(\varphi_{e,b}(g),t_{b,k}\bigr)
=
d_b\bigl(\varphi_{e,b}(h),t_{b,k}\bigr)
\qquad\text{for every } k\geq 1.
\]
By density of $\{t_{b,k}\}_{k\geq 1}$ in $G_b$, this forces
\[
\varphi_{e,b}(g)=\varphi_{e,b}(h),
\]
a contradiction. Hence $g=h$, and $d$ is a metric.
\end{proof}

We now show that the topology induced by $d$ coincides with the Bowman
topology.

\begin{lemma}\label{lem:d-continuous-bowman}
For each fixed $(f,h)\in S$, the map
\[
S\to \mathbb R,
\qquad
(e,g)\mapsto d\bigl((e,g),(f,h)\bigr)
\]
is continuous with respect to the Bowman topology.
\end{lemma}

\begin{proof}

By Corollary~\ref{cor:yeager-bowman}, the Bowman topology coincides with the
Yeager topology. Since $E(S)$ is metrizable, it suffices to prove continuity
with respect to the first-countable Yeager topology $\mathcal T_{Y_1}$.
Fix $b\in\mathcal B$ and $k\geq 1$, and define
\[
\phi_{b,k}(e,g):=
a_b(e)\,d_b\bigl(\widehat\varphi_{e,b}(g),t_{b,k}\bigr).
\]

We prove that $\phi_{b,k}$ is $\mathcal T_{Y_1}$-continuous at an arbitrary
point $(e_0,g_0)\in S$. Fix $b \in \mathcal{B}$ and $k \ge 1$, and let $\varepsilon > 0$. Set 
            $$\mathcal{O}_{\varepsilon}=\{(e, g) \in S \mid |\phi_{b,k}(e,g)- \phi_{b,k}(e_{0}, g_{0})| < \varepsilon\}.$$
         
         We distinguish three cases:
        \begin{itemize}
            \item[(i)] $b=0$. By continuity of $u \mapsto d_{0}(u, t_{0,k})$ on $G_0$, there exists an open subset $V_{0} \subseteq G_0$ with $\varphi_{e_0,0}(g_{0}) \in V_{0}$ and $$|d_{0}(u, t_{0,k})-d_{0}(\varphi_{e_0,0}(g_0),t_{0,k})| < \epsilon, \quad \text{ for every } u \in V_{0}.$$
            Let $U=E(S)$ and set $V_{f}:=(\varphi_{f,0})^{-1}(V_{0})$ (recall that each bonding map is continuous, so $V_{f}$ is an open subset of $G_f$) for each $f \in U \cap ({ \downarrow e_{0} })= \downarrow e_{0}$. If $(e,g) \in W(U, (f, V_{f}))$, then $e \in U \cap ({ \uparrow f })={ \uparrow f }$ and $\varphi_{f,e}(s) \in V_{f}=(\varphi_{f,0})^{-1}(V_{0})$, hence $\varphi_{e,0}(g)=\varphi_{f,0}(\varphi_{e,f}(g)) \in V_{0}$. Consequently $|\phi_{0,k}(e,g)-\phi_{0,k}(e_{0}, g_{0})|=|d_{0}(\varphi_{e,0}(s), t_{0,k})-d_{0}(\varphi_{e_0,0}(g_{0}), t_{0,k})| < \varepsilon$. Therefore 
            $$W(U, (f, V_{f})) \subseteq \mathcal{O}_{\epsilon}.$$
            \item[(ii)] $0 \neq b \not \ll e_{0}$. Observe that $a_{b}(e_{0})=0$, so by continuity of $a_{b}$ we have that $U:= a_{b}^{-1}((-\varepsilon, +\varepsilon))$ is an open subset of $E(S)$ containing $e_{0}$. For each $f \in U \cap ({ \downarrow e_0})$ we set $V_{f}:= G_f$. If $(e,g) \in W(U, (f, V_{f}))$, then $|\phi_{b,k}(e,g)-\phi_{b,k}(e_{0}, g_{0})| = \phi_{b,k}(e, g) \le a_{b}(e)d_{b}(\widehat{\varphi}_{e,b}(s), t_{b,k}) \le a_b(e) < \varepsilon$. Therefore 

        $$W(U, (f, V_{f})) \subseteq \mathcal{O}_{\varepsilon}.$$
        \item[(iii)] $0 \neq b \ll e_{0}$. By continuity of $a_{b}$ on $E(S)$, there exists an open set $U^{\ast} \subseteq E(S)$ with $e_{0} \in U^{\ast}$ and 
        $$|a_{b}(e)-a_{b}(e_{0})|< \varepsilon/2 \quad \text{ for every } e \in U^{\ast}.$$
        Moreover, since the topology of $E(S)$ is the Lawson topology and $b \ll e_{0}$, $\twoheaduparrow b$ is an open subset of $E(S)$ with $e_{0} \in { \twoheaduparrow b }$. Consequently $U:= U^{\ast} \cap ({ \twoheaduparrow b })$ is an open subset of $E(S)$ with $e_{0} \in U$ and such that 
        $$b \ll e \quad \text{ and } \quad|a_{b}(e) - a_{b}(e_{0})| < \varepsilon/2 \quad \text{ for every } e \in U.$$ 
        By continuity of $u \mapsto d_{b}(u, t_{b,k})$ on $G_b$, there exists an open set $V_{b} \subseteq G_b$ with $\varphi_{e_0,b}(g_{0}) \in V_{b}$ and $$|d_{b}(u, t_{b,k})-d_{b}(\varphi_{e_0,b}(g_{0}),t_{b,k})| < \varepsilon/2, \quad \text{ for every } u \in V_{b}.$$
        Set $V_{f}:=(\varphi_{f,b})^{-1}(V_{b})$ (recall that each bonding map is continuous, so $V_{f}$ is an open subset of $G_f$) for each $f \in U \cap ({ \downarrow e_{0} })$ (observe that $f \in U$ implies $b \ll f$, and so $b \le f$ by Lemma~\ref{lem:way-below-basic}). If $(e,g) \in W(U, (f, V_{f}))$, then $e \in U \cap ({ \uparrow f })$ and $\varphi_{e,f}(g) \in V_{f}=(\varphi_{f,b})^{-1}(V_f)$, hence $\varphi_{e,b}(s)=\varphi_{f,b}(\varphi_{e,f}(g)) \in V_{b}$. Consequently $|\phi_{b,k}(e,g)-\phi_{b,k}(e_{0}, g_{0})|=|a_{b}(e)d_{b}(\varphi_{e,b}(g), t_{b,k})-a_{b}(e_{0})d_{b}(\varphi_{e_0,b}(g_{0}), t_{b,k})| \le |a_{b}(e) - a_{b}(e_{0})|+|d_{b}(\varphi_{e,b}(g), t_{b,k}) - d_{b}(\varphi_{e_0,b}(g_{0}), t_{b,k})| < \varepsilon/2 + \varepsilon/2=\varepsilon$. Therefore 
        $$W(U, (f, V_{f})) \subseteq \mathcal{O}_{\varepsilon}.$$
        \end{itemize}
        Thus $\phi_{b,k}$ is continuous at $(e_0,g_0)$. Since $(e_0,g_0)$ was arbitrary,
$\phi_{b,k}$ is $\mathcal T_{Y_1}$-continuous on $S$.

\medskip

Therefore, for fixed $(f,h)\in S$, the function
\[
(e,g)\longmapsto
\bigl|\phi_{b,k}(e,g)-\phi_{b,k}(f,h)\bigr|
\]
is continuous and bounded by $1$, so the series
\[
\sum_{k\ge1}2^{-k}\bigl|\phi_{b,k}(e,g)-\phi_{b,k}(f,h)\bigr|
\]
converges uniformly and defines a continuous function.

Adding the continuous term $|a_b(e)-a_b(f)|$, we obtain that
\[
(e,g)\mapsto P_b\bigl((e,g),(f,h)\bigr)
\]
is continuous.

Finally, since $0\le P_{b_j}\le 2$, the series
\[
\sum_{j\ge1}2^{-j}P_{b_j}\bigl((e,g),(f,h)\bigr)
\]
converges uniformly and is continuous. The map
\[
(e,g)\mapsto \rho(e,f)
\]
is continuous because $\pi\colon S\to E(S)$ is continuous. Therefore
\[
(e,g)\mapsto d\bigl((e,g),(f,h)\bigr)
\]
is continuous.
\end{proof}

\begin{theorem}\label{th:metric-equals-bowman}
The metric topology induced by $d$ coincides with the Bowman topology.
\end{theorem}

\begin{proof}
By Lemma~\ref{lem:d-continuous-bowman}, every map
\[
(e,g)\mapsto d\bigl((e,g),(f,h)\bigr)
\]
is Bowman-continuous for fixed $(f,h)\in S$. Hence every open $d$-ball is
Bowman-open, and therefore the identity map
\[
\mathrm{id}\colon (S,\mathcal T_B)\to (S,\mathcal T_d)
\]
is continuous. Since $(S,\mathcal T_B)$ is compact Hausdorff by
Theorem~\ref{th:bowman}, and $(S,\mathcal T_d)$ is Hausdorff because it is
metric, the identity map is a homeomorphism. Thus
$\mathcal T_d=\mathcal T_B$.
\end{proof}

Combining this with Corollary~\ref{cor:yeager-bowman}, we obtain the main
metrizability statement announced in the introduction.

\begin{corollary}\label{cor:bowman-metrizable-explicit}
Under assumptions \textbf{(B1)}--\textbf{(B3)}, the Bowman topology on $S$ is
metrizable. More precisely, the metric defined in \eqref{eq:bowman-metric} is
compatible with the topology of $S$.
\end{corollary}

\subsection{Metrizability of the Yeager topology with isometric bonding maps}

We conclude with a simpler metrization result under a stronger compatibility
assumption on the bonding homomorphisms.

Let $S$ be a compact Hausdorff topological Clifford semigroup. Assume:
\begin{enumerate}
\item[\textbf{(Y1)}] $(E(S),\mathcal T_{E(S)})$ is metrizable, and let $\rho$ be
a compatible metric on $E(S)$;
\item[\textbf{(Y2)}] for every $e\in E(S)$ there exists a compatible metric
$d_e$ on $G_e$ such that for every $e\leq f$ the bonding homomorphism
\[
\varphi_{f,e}\colon (G_f,d_f)\to (G_e,d_e)
\]
is an isometry.
\end{enumerate}

\begin{theorem}\label{th:yeager-isometric}
Under assumptions \textbf{(Y1)}--\textbf{(Y2)}, the Yeager topology on $S$ is
metrizable.
\end{theorem}

\begin{proof}
For $(e,s),(f,t)\in S$, define
\begin{equation}\label{eq:yeager-isometric-metric}
d\bigl((e,s),(f,t)\bigr)
:=
\rho(e,f)+
d_{ef}\bigl(\varphi_{e,ef}(s),\varphi_{f,ef}(t)\bigr).
\end{equation}
This is well defined because $ef\le e,f$.
We first prove that $d$ is a metric. Non-negativity and symmetry are immediate.
If $d((e,s),(f,t))=0$, then $\rho(e,f)=0$, hence $e=f$. Therefore
\[
0=d_e(s,t),
\]
so $s=t$. Thus $d$ is definite.

For the triangle inequality, let $(e,s),(f,t),(g,u)\in S$, and put
\[
w:=efg.
\]
Then $w\le ef$, $w\le fg$, and $w\le eg$. Since the bonding homomorphisms are
isometries, we may compare all relevant distances inside the common lower group
$G_w$. Indeed,
\begin{align*}
d_{eg}\bigl(\varphi_{e,eg}(s),\varphi_{g,eg}(u)\bigr)
&=
d_w\bigl(\varphi_{eg,w}(\varphi_{e,eg}(s)),\varphi_{eg,w}(\varphi_{g,eg}(u))\bigr)\\
&=
d_w\bigl(\varphi_{e,w}(s),\varphi_{g,w}(u)\bigr)\\
&\le
d_w\bigl(\varphi_{e,w}(s),\varphi_{f,w}(t)\bigr)
+
d_w\bigl(\varphi_{f,w}(t),\varphi_{g,w}(u)\bigr)\\
&=
d_{ef}\bigl(\varphi_{e,ef}(s),\varphi_{f,ef}(t)\bigr)
+
d_{fg}\bigl(\varphi_{f,fg}(t),\varphi_{g,fg}(u)\bigr).
\end{align*}
Combining this with the triangle inequality for $\rho$ yields
\[
d\bigl((e,s),(g,u)\bigr)
\le
d\bigl((e,s),(f,t)\bigr)+d\bigl((f,t),(g,u)\bigr).
\]
Hence $d$ is a metric.
We next prove that the metric topology is contained in the Yeager topology.
Let $B_r(p)$ be an open $d$-ball and let $(f,t)\in B_r(p)$. Choose
$\varepsilon>0$ such that
\[
d\bigl((f,t),p\bigr)+2\varepsilon<r.
\]
Let
\[
U:=\{e\in E(S):\rho(e,f)<\varepsilon\}.
\]
For each $a\in U\cap\downarrow f$, define
\[
V_a:=
\{g\in G_a:d_a(g,\varphi_{f,a}(t))<\varepsilon\}.
\]
Then $V_a$ is open in $G_a$, and $(f,t)\in W(U,(a,V_a))$.
Now let $(e,s)\in W(U,(a,V_a))$. Then $e\in U\cap\uparrow a$ and
\[
\varphi_{e,a}(s)\in V_a,
\]
so
\[
d_a\bigl(\varphi_{e,a}(s),\varphi_{f,a}(t)\bigr)<\varepsilon.
\]
Since $a\le ef$, the map $\varphi_{ef,a}$ is an isometry, and by functoriality
\[
\varphi_{ef,a}\circ\varphi_{e,ef}=\varphi_{e,a},
\qquad
\varphi_{ef,a}\circ\varphi_{f,ef}=\varphi_{f,a}.
\]
Therefore
\[
d_{ef}\bigl(\varphi_{e,ef}(s),\varphi_{f,ef}(t)\bigr)
=
d_a\bigl(\varphi_{e,a}(s),\varphi_{f,a}(t)\bigr)
<\varepsilon.
\]
Since also $\rho(e,f)<\varepsilon$, it follows that
\[
d\bigl((e,s),(f,t)\bigr)<2\varepsilon.
\]
Hence, by the triangle inequality,
\[
d\bigl((e,s),p\bigr)
\le
d\bigl((e,s),(f,t)\bigr)+d\bigl((f,t),p\bigr)
<
2\varepsilon+d\bigl((f,t),p\bigr)
<r.
\]
Thus
\[
W(U,(a,V_a))\subseteq B_r(p),
\]
so every metric ball is Yeager-open.
Therefore the identity map
\[
\mathrm{id}\colon (S,\mathcal T_Y)\to (S,\mathcal T_d)
\]
is continuous. Since $(S,\mathcal T_Y)$ is compact Hausdorff by
Theorem~\ref{th:yeager}, and $(S,\mathcal T_d)$ is Hausdorff, the identity map
is a homeomorphism. Hence the Yeager topology is metrizable.
\end{proof}

\section{Hilbert's fifth problem for topological Clifford semigroups}\label{sec:hilbert}

Hilbert's fifth problem, in its modern formulation, asks whether every locally
Euclidean topological group is necessarily a Lie group; see, for instance,
\cite{TaoBook}. The answer is affirmative by the work of Gleason and of
Montgomery--Zippin \cite{Gleason,MontgomeryZippinAnnals}.

\begin{theorem}[Hilbert's fifth problem]\label{th:hilbert5}
Every locally Euclidean topological group is a Lie group.
\end{theorem}

Recall that a topological group is NSS if it has a neighbourhood of the
identity containing no nontrivial subgroup. We shall also use the following classical criteria. 

\begin{theorem}[Gleason--Yamabe {\cite{GleasonDuke,Yamabe}}]\label{th:gleason-yamabe}
Every locally compact NSS topological group is a Lie group.
\end{theorem}

\begin{theorem}[Gleason--Palais {\cite{GleasonPalais}}]\label{th:gleason-palais}
Let $G$ be a topological group which is locally arcwise connected and admits a
neighbourhood of the identity that can be continuously and injectively mapped
into a finite-dimensional metric space. Then $G$ is a Lie group.
\end{theorem}

These theorems suggest asking whether analogous statements hold for topological
semigroups; for background on this direction we also refer to
\cite{HofmannHilbert5}.
 For general semigroups this question is subtle, and even in the
Clifford setting one should not expect a global Lie structure on the whole
space. Indeed, the semilattice of idempotents may carry singular topological
behaviour, and different maximal subgroups may have different dimensions.

The appropriate Clifford-semigroup analogue is therefore local in nature: under
which assumptions on a topological Clifford semigroup are the maximal subgroups
forced to be Lie groups?

We begin by fixing the relevant notions.

\begin{definition}[Weakly locally Euclidean]
A topological space $X$ is said to be \emph{weakly locally Euclidean} if for
every $x\in X$ there exist an integer $n(x)\in\mathbb N$, an open neighbourhood
$U_x$ of $x$, and a homeomorphism
\[
\varphi_x\colon U_x\to V_x\subseteq \mathbb R^{n(x)},
\]
where $V_x$ is open in $\mathbb R^{n(x)}$.
\end{definition}

In particular, the local Euclidean dimension is allowed to depend on the point.

\begin{definition}[Weakly locally Euclidean at the idempotents]
Let $S$ be a topological Clifford semigroup. We say that $S$ is
\emph{weakly locally Euclidean at the idempotents} if for every idempotent
$e\in E(S)$ there exist an integer $n(e)\in\mathbb N$, an open neighbourhood
$U_e$ of $e$, and a homeomorphism
\[
\varphi_e\colon U_e\to V_e\subseteq \mathbb R^{n(e)},
\]
where $V_e$ is open in $\mathbb R^{n(e)}$.
\end{definition}

\begin{definition}[Lie type]\label{def:lie-type}
A topological Clifford semigroup $S$ is said to be \emph{of Lie type} if for
every idempotent $e\in E(S)$, the maximal subgroup $G_e$, endowed with the
subspace topology inherited from $S$, is a finite-dimensional Lie group.
\end{definition}

The next result is the natural Hilbert-fifth-type statement for strong
semilattices of topological groups.

\begin{theorem}\label{th:hilbert-clifford}
Let $S$ be a strong semilattice of topological groups. If $S$ is weakly locally
Euclidean at the idempotents, then $S$ is of Lie type.
\end{theorem}

\begin{proof}
By Proposition~\ref{prop:mp-equivalences}, each maximal subgroup $G_e$ is open
in $S$. Let $e\in E(S)$. Since $S$ is weakly locally Euclidean at the
idempotents, there exists an open neighbourhood $U_e$ of $e$ in $S$ which is
homeomorphic to an open subset of some Euclidean space. Then
\[
U_e\cap G_e
\]
is an open neighbourhood of $e$ in the topological group $G_e$, and is itself
homeomorphic to an open subset of a Euclidean space. Hence $G_e$ is locally
Euclidean. By Theorem~\ref{th:hilbert5}, $G_e$ is a Lie group. Since $e$ was
arbitrary, $S$ is of Lie type.
\end{proof}

We next obtain a Gleason--Yamabe-type criterion.
In the Clifford setting, the
relevant notion is the following.

\begin{definition}[NSS topological Clifford semigroup]
A topological Clifford semigroup $S$ is said to be \emph{NSS} if every maximal
subgroup $G_e$ is an NSS topological group.
\end{definition}

\begin{theorem}\label{th:nss-clifford}
Let $S$ be a locally compact NSS topological Clifford semigroup such that $E(S)$ is $T_2$ (with the induced topology). Then $S$ is of
Lie type.
\end{theorem}

\begin{proof}
Let $e\in E(S)$. Since $S$ is locally compact and $G_e$ is closed in $S$, the
group $G_e$ is locally compact in the subspace topology. By assumption, $G_e$
is NSS. Therefore Theorem~\ref{th:gleason-yamabe} applies and yields that $G_e$
is a Lie group. Since $e$ was arbitrary, $S$ is of Lie type.
\end{proof}

We also record a Gleason--Palais-type criterion.

\begin{theorem}\label{th:arc-connected-clifford}
Let $S$ be a topological Clifford semigroup. Assume that:
\begin{enumerate}
\item $S$ is weakly locally Euclidean at the idempotents;
\item for every $e\in E(S)$, the maximal subgroup $G_e$ is locally
arcwise connected.
\end{enumerate}
Then $S$ is of Lie type.
\end{theorem}

\begin{proof}
Let $e\in E(S)$. By assumption, there exist an open neighbourhood $U_e$ of $e$
in $S$ and a homeomorphism
\[
\varphi_e\colon U_e\to V_e\subseteq \mathbb R^{n(e)},
\]
where $V_e$ is open. Restricting $\varphi_e$ to $U_e\cap G_e$, we obtain a
continuous injective map
\[
U_e\cap G_e \longrightarrow \mathbb R^{n(e)}.
\]
Since $U_e\cap G_e$ is an open neighbourhood of the identity in the topological
group $G_e$, and $G_e$ is locally arcwise connected by hypothesis,
Theorem~\ref{th:gleason-palais} applies and shows that $G_e$ is a Lie group.
Since $e$ was arbitrary, $S$ is of Lie type.
\end{proof}

The previous results show that several classical Lie-group criteria extend
naturally to the level of maximal subgroups of a Clifford semigroup. This is
the correct analogue of Hilbert's fifth problem in the present setting.

We conclude with a simple example illustrating why one should not expect a
global smooth or Lie structure on an arbitrary topological Clifford semigroup.

\begin{example}\label{ex:min-example}
Consider $\mathbb R^2$ endowed with the operation
\[
(e,s)\cdot(f,t):=(\min\{e,f\},\,s+t).
\]
Then $\mathbb R^2$ is a trivial topological Clifford semigroup. Its idempotents
form the semilattice
\[
E(\mathbb R^2)=\mathbb R\times\{0\},
\]
whose multiplication is induced by the minimum operation on the first
coordinate. This example shows that even very simple topological Clifford semigroups may
fail to admit a compatible manifold-type structure in the idempotent direction,
so that a global Lie-semigroup structure is not to be expected in general.
\end{example}

\section{$C^1$-regularity at idempotents}\label{sec:c1}

As already noted in the previous section, differentiability for semigroups is
considerably subtler than for groups. In a topological group, local smooth
charts can be transported by translations, which are homeomorphisms. In a semigroup this is no longer true, in general, and therefore even a smooth local
description near one point need not propagate to nearby points.

For this reason, one should distinguish between two different settings:
\begin{enumerate}
\item a \emph{global} differentiable-semigroup structure on the whole space;
\item a \emph{local} differentiability condition imposed only near the
idempotents.
\end{enumerate}
The first point of view goes back to Holmes, whose results on differentiable
semigroups provide useful information on the local structure of idempotents. The
second point of view is the one relevant for the rigidity results proved in this
paper.

\subsection{Holmes' differentiable semigroups}

We begin by recalling Holmes' global notion.

\begin{definition}\label{def:holmes-semigroup}
A \emph{$C^1$-differentiable semigroup} is a topological semigroup $(S,\cdot)$
such that:
\begin{enumerate}
\item $S$ is a $C^1$ manifold modelled on a Banach space;
\item the multiplication map
\[
m\colon S\times S\to S,\qquad m(x,y)=xy,
\]
is of class $C^1$.
\end{enumerate}
\end{definition}

Holmes proved several structural results on the idempotent set of a
$C^1$-differentiable semigroup (see \cite{Holmes}). For our purposes, the
following local consequence is sufficient.

\begin{proposition}\label{prop:holmes-control}
Let $S$ be a $C^1$-differentiable semigroup, and let $e\in E(S)$. Then there
exists an open neighbourhood $U_e$ of $e$ such that for every
$f\in E(S)\cap U_e$ and every $x\in U_e$, one has
\[
fxf\in G_f,
\]
where $G_f$ denotes the maximal subgroup at $f$.
\end{proposition}

As a first consequence, one obtains discreteness of the idempotent set whenever
idempotents commute.

\begin{corollary}\label{cor:holmes-discrete-idempotents}
Let $S$ be a $C^1$-differentiable semigroup such that idempotents commute
pairwise. Then $E(S)$ is discrete in $S$.
\end{corollary}

\begin{proof}
Fix $e\in E(S)$, and let $U_e$ be as in
Proposition~\ref{prop:holmes-control}. We claim that
\[
E(S)\cap U_e=\{e\}.
\]
Let $f\in E(S)\cap U_e$. Applying Proposition~\ref{prop:holmes-control} with
$x=e$, we get
\[
fef\in G_f.
\]
Since idempotents commute, this becomes
\[
ef=fe=ffe = fef\in G_f.
\]
On the other hand, again by Proposition~\ref{prop:holmes-control}, now with the
roles of $e$ and $f$ exchanged,
\[
ef = efe\in G_e.
\]
Thus $ef$ belongs both to $G_e$ and to $G_f$. Since the maximal subgroups are
pairwise disjoint, it follows that $e=f$. Therefore
\[
E(S)\cap U_e=\{e\},
\]
so $e$ is isolated in $E(S)$. Since $e$ was arbitrary, $E(S)$ is discrete.
\end{proof}

In the finite-dimensional inverse-semigroup setting, this already implies a Lie
structure on the maximal subgroups.

\begin{corollary}\label{cor:finite-dim-inverse-lie}
Let $S$ be a finite-dimensional $C^1$-differentiable Clifford semigroup. Then,
for every $e\in E(S)$, the maximal subgroup $G_e$ is an open submanifold of
$S$. In particular, each $G_e$ is a finite-dimensional Lie group and $S$ is a strong semilattice of Lie groups.
\end{corollary}
\begin{proof}
By Corollary~\ref{cor:holmes-discrete-idempotents}, the idempotent set $E(S)$
is discrete in $S$. Since inversion and multiplication are continuous, the map
\[
\pi\colon S\to E(S),\qquad \pi(x)=xx^{-1}
\]
is continuous and, for each $e\in E(S)$ one has
\[
G_e=\pi^{-1}(\{e\}). 
\]
Because $E(S)$ is discrete, the singleton $\{e\}$ is open in $E(S)$, hence
$G_e$ is open in $S$. Since $S$ is a finite-dimensional $C^1$ manifold, $G_e$
is an open finite-dimensional manifold. Being also a topological group, it is a
finite-dimensional Lie group by Theorem~\ref{th:hilbert5}. By Proposition~\ref{prop:mp-equivalences}, $S$ is a strong semilattice
of Lie groups.
\end{proof}

%

The preceding results are global in nature, since they assume that the whole
space carries a manifold structure and that multiplication is globally $C^1$.
For topological Clifford semigroups this is often too restrictive. In
particular, even if the maximal subgroups are Lie groups, there is no reason for
the whole space to admit a compatible manifold structure.

This motivates the following weaker, purely local notion, introduced in the following subsection.

\subsection{$C^1$-regularity at idempotents}

\begin{definition}\label{def:c1-idempotents}
Let $S$ be a topological Clifford semigroup with multiplication
\[
m\colon S\times S\to S,\qquad m(x,y)=xy.
\]
We say that $S$ is \emph{$C^1$ at the idempotents} if for every
$e\in E(S)$ there exist:
\begin{enumerate}
\item a Banach space $X_e$;
\item open neighbourhoods
\[
e\in U_e\subseteq W_e\subseteq S;
\]
\item a homeomorphism
\[
\varphi_e\colon W_e\to V_e\subseteq X_e,
\qquad
\varphi_e(e)=0,
\]
where $V_e$ is open in $X_e$;
\end{enumerate}
such that
\[
m(U_e\times U_e)\subseteq W_e,
\]
and the coordinate representation of multiplication
\[
\widetilde\mu_e\colon \varphi_e(U_e)\times \varphi_e(U_e)\to X_e,
\qquad
\widetilde\mu_e(u,v):=
\varphi_e\!\bigl(\varphi_e^{-1}(u)\varphi_e^{-1}(v)\bigr),
\]
is of class $C^1$ in the Fr\'echet sense.
\end{definition}

This condition only requires differentiability of the multiplication near pairs
$(e,e)$ with $e\in E(S)$; no global manifold structure on $S$ is assumed.

The main result of this section is the following rigidity statement.

\begin{theorem}\label{th:c1-idempotents-discrete}
Let $S$ be a topological Clifford semigroup which is $C^1$ at the idempotents.
Then the idempotent semilattice $E(S)$ is discrete in $S$. In particular, $S$
is a strong semilattice of topological groups.
\end{theorem}

\begin{proof}
By Proposition~\ref{prop:mp-equivalences}, it is enough to prove that $E(S)$ is
discrete.
Let $e\in E(S)$. Choose a Banach space $X_e$, neighbourhoods
\[
e\in U_e\subseteq W_e\subseteq S,
\]
and a chart
\[
\varphi_e\colon W_e\to V_e\subseteq X_e,
\qquad \varphi_e(e)=0,
\]
as in Definition~\ref{def:c1-idempotents}. Set
\[
V'_e:=\varphi_e(U_e),
\]
which is an open neighbourhood of $0$ in $X_e$.
Define the local squaring map
\[
F\colon U_e\to W_e,
\qquad
F(x):=x^2.
\]
By assumption, $F$ is well defined, and its coordinate expression
\[
f:=\varphi_e\circ F\circ \varphi_e^{-1}\colon V'_e\to X_e
\]
is of class $C^1$. Since $e^2=e$, one has
\[
f(0)=0.
\]

Now a point $x\in U_e$ is idempotent if and only if
$x^2=x$.
In coordinates, this means that for
\[
u=\varphi_e(x)\in V'_e
\]
one has
\[
f(u)=u.
\]
Thus idempotents in $U_e$ correspond exactly to fixed points of $f$.

Define
\[
H\colon V'_e\to X_e,
\qquad
H(u):=f(u)-u.
\]
We claim that $0$ is an isolated zero of $H$.

To prove this, let
\[
\mu_e\colon V'_e\times V'_e\to X_e
\]
be the coordinate expression of multiplication:
\[
\mu_e(u,v):=
\varphi_e\!\bigl(\varphi_e^{-1}(u)\varphi_e^{-1}(v)\bigr).
\]
Then
\[
f(u)=\mu_e(u,u).
\]
Since $\mu_e$ is $C^1$, its derivative at $(0,0)$ is a continuous linear map
\[
D\mu_e(0,0)\colon X_e\times X_e\to X_e.
\]
Define continuous linear operators
\[
L,R\in\mathcal L(X_e,X_e)
\]
by
\[
L(h):=D\mu_e(0,0)(h,0),
\qquad
R(h):=D\mu_e(0,0)(0,h).
\]
Then, for every $h,k\in X_e$,
\[
D\mu_e(0,0)(h,k)=L(h)+R(k).
\]
In particular, by the chain rule,
\[
Df(0)h=D\mu_e(0,0)(h,h)=L(h)+R(h).
\]

Now consider the coordinate versions of left and right multiplication by $e$:
\[
\lambda(u):=\mu_e(0,u),
\qquad
\rho(u):=\mu_e(u,0).
\]
Since $e$ is idempotent, one has
\[
\lambda\circ\lambda=\lambda,
\qquad
\rho\circ\rho=\rho.
\]
Differentiating at $0$ gives
\[
R^2=R,
\qquad
L^2=L,
\]
so both $L$ and $R$ are projections.
Because $S$ is Clifford, idempotents are central. In particular,
$ex=xe$
for all $x$ near $e$, and therefore
$\lambda=\rho$.
Hence
\[
L=R=:P,
\]
where $P$ is a projection. It follows that
\[
Df(0)=2P.
\]
Therefore
\[
DH(0)=Df(0)-I=2P-I.
\]
Since $P^2=P$, we compute
\[
(2P-I)^2=4P^2-4P+I=I.
\]
Thus $2P-I$ is invertible, with inverse itself. In particular, $DH(0)$ is a
linear isomorphism.

By the Banach inverse function theorem, there exists an open neighbourhood
\[
\Omega_e\subseteq V'_e
\]
of $0$ such that
\[
H|_{\Omega_e}
\]
is a homeomorphism onto an open neighbourhood of $0$ in $X_e$. Since $H(0)=0$,
it follows that $0$ is the unique zero of $H$ in $\Omega_e$.
Set
\[
U'_e:=\varphi_e^{-1}(\Omega_e).
\]
Then $U'_e$ is an open neighbourhood of $e$ in $S$, and the only idempotent it
contains is $e$. Hence
\[
E(S)\cap U'_e=\{e\}.
\]
Since $e$ was arbitrary, $E(S)$ is discrete in $S$.
The final assertion follows from Proposition~\ref{prop:mp-equivalences}.
\end{proof}

The preceding theorem shows that local differentiability at idempotents has a
global structural consequence: it forces the topology of $S$ to split as the
disjoint union topology of its maximal subgroups.

\begin{corollary}\label{cor:c1-plus-wle-lie}
Let $S$ be a topological Clifford semigroup which is $C^1$ at the idempotents
and weakly locally Euclidean at the idempotents. Then $S$ is of Lie type.
\end{corollary}

\begin{proof}
By Theorem~\ref{th:c1-idempotents-discrete}, $S$ is a strong semilattice of
topological groups. Since $S$ is weakly locally Euclidean at the idempotents,
Theorem~\ref{th:hilbert-clifford} applies and yields that $S$ is of Lie type.
\end{proof}

\subsection{Finite-dimensional consequences}

We now combine the rigidity theorem with the Hilbert-fifth-type results proved
in Section~\ref{sec:hilbert}.

\begin{corollary}\label{cor:c1-idempotents-lie}
Let $S$ be a topological Clifford semigroup which is $C^1$ at the idempotents,
and assume that each chart in Definition~\ref{def:c1-idempotents} takes values
in a finite-dimensional space. Then $S$ is a strong semilattice of
finite-dimensional Lie groups.
\end{corollary}

\begin{proof}
By Theorem~\ref{th:c1-idempotents-discrete}, the semilattice $E(S)$ is
discrete, hence $S$ is a strong semilattice of topological groups by
Proposition~\ref{prop:mp-equivalences}.
Since each chart at an idempotent takes values in a finite-dimensional Banach
space, the space $S$ is weakly locally Euclidean at the idempotents. Therefore
Theorem~\ref{th:hilbert-clifford} applies and shows that each maximal subgroup
$G_e$ is a finite-dimensional Lie group. Hence $S$ is a strong semilattice of
finite-dimensional Lie groups.
\end{proof}

The converse also holds.

\begin{proposition}\label{prop:lie-implies-c1}
Let $S$ be a strong semilattice of finite-dimensional Lie groups. Then $S$ is
$C^\infty$ at the idempotents. In particular, $S$ is $C^1$ at the idempotents.
\end{proposition}

\begin{proof}
By Proposition~\ref{prop:mp-equivalences}, each maximal subgroup $G_e$ is open
in $S$. Fix $e\in E(S)$. Since $G_e$ is a finite-dimensional Lie group, there
exist an open neighbourhood
\[
e\in U_e\subseteq G_e
\]
and a smooth chart
\[
\psi_e\colon U_e\to \Omega_e\subseteq \mathbb R^{n_e}.
\]
Shrinking $U_e$ if necessary, we may assume that
\[
U_e\cdot U_e\subseteq G_e
\]
and that the coordinate representation of the group multiplication is smooth.
Since $G_e$ is open in $S$, the set $U_e$ is also open in $S$. Thus
Definition~\ref{def:c1-idempotents} is satisfied, with $W_e=U_e$ and
$\varphi_e=\psi_e$. Hence $S$ is $C^\infty$ at the idempotents.
\end{proof}

Combining the two results, we obtain the finite-dimensional
characterization announced in the introduction.

\begin{theorem}\label{th:c1-idempotents-characterization}
For a topological Clifford semigroup $S$, the following are equivalent:
\begin{enumerate}
\item $S$ is $C^1$ at the idempotents with finite-dimensional charts;
\item $S$ is a strong semilattice of finite-dimensional Lie groups.
\end{enumerate}
\end{theorem}

\begin{proof}
The implication (1)$\Rightarrow$(2) is Corollary~\ref{cor:c1-idempotents-lie},
while (2)$\Rightarrow$(1) is Proposition~\ref{prop:lie-implies-c1}.
\end{proof}

\begin{remark}\label{rem:banach-limits}
The passage from ``topological group'' to ``Lie group'' in the above corollaries
depends crucially on finite dimensionality, through Hilbert's fifth problem. In
infinite dimensions no comparable general theorem is available. Thus
Theorem~\ref{th:c1-idempotents-discrete} should be viewed primarily as a
topological rigidity result: local $C^1$ regularity at the idempotents forces
discreteness of the idempotent semilattice, independently of any
finite-dimensional Lie theory.
\end{remark}

\section*{Ackowlegments}
The authors gratefully acknowledge the support of the following projects and funding sources:
\begin{itemize}
\item the Italian Ministry of Education, University and Research through the PRIN 2022 project DeKLA (``Developing Kleene Logics and their Applications'', project code: 2022SM4XC8). 

\item the Fondazione di Sardegna for the support received by the project MAPS (grant number F73C23001550007)  LOMEA and  ProBiki.
\item the INDAM GNSAGA (Gruppo Nazionale per le Strutture Algebriche, Geometriche e loro Applicazioni).
\end{itemize}

\end{document}